\documentclass[12pt]{amsart}
\usepackage{amsmath,amssymb,amsthm,amscd,latexsym}
\usepackage[mathscr]{eucal}
\usepackage[pdftex]{graphicx}
\usepackage{siunitx}
\usepackage{empheq}
\usepackage{breqn}
\usepackage{booktabs}
\usepackage{hyperref}
\usepackage[margin=1in]{geometry}
\usepackage{graphics}
\usepackage{setspace}
\usepackage{paralist}
\usepackage{cancel}
\usepackage{multirow}
\usepackage{color}
\usepackage{mathtools}
\usepackage{enumitem}
\usepackage{marginnote}
\reversemarginpar
\usepackage{tikz}
\usetikzlibrary{matrix,arrows,positioning}
\usetikzlibrary{matrix,decorations.pathreplacing}
\tikzset{
    >=stealth',
    punkt/.style={
           rectangle,
           rounded corners,
           draw=black, very thick,
           text width=6.5em,
           minimum height=2em,
           text centered},
    pil/.style={
           ->,
           thick,
           shorten <=2pt,
           shorten >=2pt,}
}
\newcommand{\abs}[1]{\left| #1 \right|}
\newcommand{\set}[1]{\left\{ #1 \right\}}

\newcommand{\Z}{\mathbb{Z}}
\newcommand{\R}{\mathbb{R}}

\newcommand{\ga}{\alpha}
\newcommand{\gb}{\beta}
\newcommand{\gr}{\gamma}
\newcommand{\gG}{\Gamma}
\newcommand{\gd}{\delta}

\newcommand{\gth}{\theta}

\newcommand{\gf}{\varphi}

\newcommand{\p}{\pi}
\newcommand{\s}{\sigma}

\newcommand{\gk}{\kappa}

\newcommand{\ve}{\varepsilon}
\newcommand{\go}{\omega}
\newcommand{\gO}{\Omega}

\newcommand{\del}{\partial}
\newcommand{\oo}{\infty}

\newcommand{\scr}[1]{\mathscr{#1}}

\renewcommand{\le}{\leqslant}
\renewcommand{\ge}{\geqslant}
\newcommand{\scrA}{\scr{A}}

\renewcommand{\ss}{\mathfrak{S}}

\renewcommand{\pmod}[1]{\ (\text{mod}\ #1)}

\theoremstyle{plain}
\newtheorem{lemma}{Lemma}[section]

\newtheorem{thm}{Theorem}

\theoremstyle{definition}

\theoremstyle{remark}

\begin{document}

\title{Almost-Prime Polynomials with prime arguments}

\author{P.-H. Kao}

\address{Central Michigan University}

\email{kao1p@cmich.edu}

\begin{abstract}
    We improve Irving's method of the double-sieve~\cite{Irv} by using the DHR sieve. By extending the upper and lower bound sieve functions into their respective non-elementary ranges, we are able to make improvements on the previous records on the number of prime factors of irreducible polynomials at prime arguments. In particular, we prove that irreducible quadratics over $\Z$ satisfying necessary local conditions are $P_4$ infinitely often.
\end{abstract}

\maketitle

\section{Introduction}

V.\ Bouniakowsky conjectured~\cite{Bun} that under suitable hypotheses, any irreducible polynomial $f \in \Z[x]$ is prime for infinitely many integral arguments $n$. The conjecture is beyond the reach of current technology. Sieve methods, however, can provide approximations to show that such an $f(n)$ is infinitely often a $P_r$-number, that is,  $\gO(f(n)) \le r$ for infinitely many $n$. One can similarly consider bounds on the number of prime factors of $f(p)$, where $p$ is a prime.

The first result in this direction was due to R.\ J.\ Miech~\cite{Mie}, who used Brun's sieve and Renyi's equidistribution theorem (which was a precursor of the Bombieri--Vinogradov theorem) to show that under the hypotheses of the Bouniakowsky conjecture, and if $k$ is the degree of $f$, then $f(p)$ is infinitely often a $P_{ck}$ for some fixed constant $c$. H.-E.\ Richert~\cite{Ric} later reduced $ck$ to $2k+1$. The main ingredients of Richert's proof consist of a weighted Selberg's sieve and the Bombieri--Vinogradov theorem. The weights in Richert's proof were restricted to the primes $p < x^{1/2}$. In his much celebrated paper~\cite{Chen}, Chen showed that $f(p)$ is infinitely often a $P_2$-number for $k=1$. Except for Chen's achievement, the subject remained dormant for nearly fifty years until the improvements made by A.\ J.\ Irving~\cite{Irv}. Irving's innovation was to observe that one could incorporate weights with primes $> x^{1/2}$ by applying a two-dimensional sieve to the sequence $\set{nf(n)}$. By appealing to the one-and two-dimensional beta sieves~\cite[Theorem 11.13]{FI}, he was able to show that $f(p)$ is infinitely often a $P_{r_0(k)}$, where $r_0(k)$ is given by the table below.

\begin{table}[h]
\centering
\renewcommand{\arraystretch}{1.2}
\begin{tabular}{@{}crrrrrrrrr@{}}\toprule\label{T:1}
         $k$ & 2 & 3 & 4 &  5 &  6 &  7 &  8 &  9 & 10\\ \hline
    $r_0(k)$ & 5 & 6 & 8 & 10 & 11 & 12 & 14 & 15 & 16\\ \bottomrule
\end{tabular}
\end{table}

For large $k$, he showed that
\[
    r_0(k) = k + c_0 \log k + O(1),
\]
with $c_0 = 3.120....$

As one of the concluding remarks in~\cite{Irv}, Irving suggested the possibility of improvements by applying the Diamond--Halberstam--Richert sieve in place of the the beta sieve. We carry out his suggestion and extend all sifting functions into their respective non-elementary ranges. We prove the following theorem.

\begin{thm}\label{T:K}
  	Let $f(x) \in \Z[x]$ be an irreducible polynomial of degree $k$ with a positive leading coefficient. Suppose further that for all primes $p$, we have
	\[
		\#\{ a \pmod{p} : (a,p) = 1 \text{ and } f(a) \equiv 0 \pmod{p}\} < p - 1.
	\]
    Then there exists a function $r(k)$ such that for sufficiently large $x$ and for $r \ge k$, we have
    \[
        \#\{ x < p \le 2x : f(p) \in P_r \} \gg_{f,r} \frac{x}{(\log x)^2}.
    \]
    Moreover, for sufficiently large $k$, $f(p)$ is infinitely often a $P_{r(k)}$-number for
	\[
		r(k) = k + c \log k + O(1)
	\]
	with $c = 1.18751...$. For small values of $k$, the same result is true with values of $r(k)$ provided in Table~\ref{T:1}.

    \begin{table}[h]
    \centering
	\renewcommand{\arraystretch}{1.2}
    \begin{tabular}{@{}crrrrrrrrr@{}}\toprule\label{T:1}
        $k$ & 2 & 3 & 4 & 5 & 6 & 7 & 8 & 9 & 10\\ \hline
		$r(k)$ & 4 & 6 & 7 & 9 & 10 & 11 & 12 & 13 & 15\\ \bottomrule
    \end{tabular}
	\label{Table:1}
	\vspace{6pt}
	\caption{Values of $r(k)$ for small $k$}
	\end{table}
\end{thm}

We remark that Theorem~\ref{T:K} improves Irving's results except for $k = 3$. In particular, we are able to reduce $r(2)$ from 5 to 4.

Wu and Xi very recently announced the result that irreducible quadratics at prime arguments are $P_4$ infinitely often~\cite{WX}. They approached the problem by employing composition of the linear sieves rather than appealing to the two-dimensional sieve as Irving did. On the other hand, Wu and Xi extended the level of distribution up to $x^{0.79}$. It is unclear if their argument could be generalized to polynomials of higher degree.

\section{Preliminaries}

Throughout this paper, the letters $p$ and $q$ are used to denote primes. For any natural number $n$, we let $\go(n)$ denote the number of distinct prime factors of $n$ not counting multiplicities; $\gO(n)$ count the number of prime factors of $n$ with multiplicities. $\gf(n)$ denotes Euler's totient function. $f \sim g$ denotes $f/g \to 1$ as $x \to \oo$; $f = o(g)$ denotes $f/g \to 0$ as $x \to \oo$. We write $f = O(g)$ or, equivalently, $f \ll g$ if $\abs{f} \le Cg$ for some real number $C$. We remark that all implied constants will depend on $f$ unless otherwise indicated.

The sequence of interest in this paper is
\begin{equation}\label{E:scrA}
	\scrA = \set{f(p) : p \in (x,2x]},
\end{equation}
where $f$ is irreducible over $\Z$ and has nonzero constant term. Denote $k = \deg f$ and let $N = \max \scr{A}$. Then
\[
	N \sim a_k x^k,
\]
where $a_k > 0$ is the leading coefficient of $f$. We denote $|\scrA| = X$, whence by the Prime Number Theorem,
\[
    X \sim \frac{x}{\log x}.
\]

We introduce the following arithmetic functions that are necessary in the sieve constructions. For squarefree $d$, let
\begin{equation}\label{E:nu_1}
	\nu_1(d) = \#\{ a \pmod{d} : (a,d) = 1 \text{ and } f(a) \equiv 0 \pmod{d} \}
\end{equation}
and
\begin{equation}\label{E:nu_2}
	\nu_2(d) = \#\{ a \pmod{d} : af(a) \equiv 0 \pmod{d} \}.
\end{equation}
We also provide Mertens' type estimates for $\nu_1$ and $\nu_2$, as they are frequently needed in sieve estimates.
\begin{lemma}
    For $x \ge 2$, we have
    \begin{align}
        \sum_{p \le x} \frac{\nu_1(p) \log p}{p-1} &= \log x + O(1),\label{E:nu_1_sum}\\
        \prod_{p \le x} \Bigl( 1 - \frac{\nu_1(p)}{p-1} \Bigr) &= \frac{e^{-\gr}\ss(f)}{\log x} (1 + o(1)),\label{E:nu_1_prod}\\
        \sum_{p \le x} \frac{\nu_2(p) \log p}{p} &= 2 \log x + O(1),\label{E:nu_2_sum}\\
        \prod_{p \le x} \Bigl( 1 - \frac{\nu_2(p)}{p} \Bigr) &= \frac{e^{-2\gr} \ss(f)}{(\log x)^2} (1 + o(1)),\label{E:nu_2_prod}
    \end{align}
    where $\ss(f)$ is defined by
    \[
        \ss(f) = 2 \prod_{p > 2} \left( 1 - \frac{\nu_1(p)}{p-1} \right) \left( 1 - \frac{1}{p} \right)^{-1}.
    \]
\end{lemma}

\begin{proof}
    See the proof of~\cite[Lemma 2.1]{Irv}.
\end{proof}

We adopt the following standard sieve theoretic notations and hypotheses. Let
\[
    S(\scrA, z) = \sum_{\substack{n \in \scrA\\ (n,P(z))=1}} 1,
\]
where
\[
    P(z) = \prod_{\substack{\text{$p$ prime}\\ p \le z}} p.
\]
Define
\[
    \scrA_d = \{ n \in \scrA : n \equiv 0 \pmod{d} \}.
\]
We assume that there exists some non-negative multiplicative function $\nu(\cdot)$ such that
\[
    0 \le \nu(p) < p,
\]
for primes $p$ in relevant range according to applications so that the error term
\[
    R_d = \abs{\scrA_d} - \frac{\nu(d)}{d}X,
\]
whenever $d|P(z)$, is on average small, over some restricted range of values of $d$. We also assume that there exist $\gk > 1$ and $A > 1$ such that
\begin{equation}\label{E:gO_condition}
    \prod_{w_1 \le p < w_2} \biggl( 1 - \frac{\nu(p)}{p} \biggr)^{-1} \le \biggl( \frac{\log w_2}{\log w_1} \biggr)^{\gk} \biggl( 1 + \frac{A}{\log w_1} \biggr),
\end{equation}
where $2 \le w_1 \le w_2$. An immediate consequence of~\eqref{E:gO_condition} is that
\[
    \prod_{w_1 \le p < w_2} \biggl( 1 - \frac{\nu(p)}{p} \biggr)^{-1} \ll (\log w_2)^{\gk}.
\]
One can easily see that equations~\eqref{E:nu_1_prod} and~\eqref{E:nu_2_prod} are examples of~\eqref{E:gO_condition}.

Let
\[
    E(x,d) = \max_{\substack{1 \le m \le d\\ (m,d)=1}} \Bigl| \p(x;d,m) - \frac{X}{\gf(d)} \Bigr|,
\]
then the Bombieri--Vinogradov theorem~\cite[Chapter 28]{Dav} states that for any constant $A > 0$, there exists a number $B = B(A)$ such that
\[
    \sum_{d \le x^{1/2} (\log x)^{-B}} E(x,d) \ll \frac{x}{(\log x)^A}.
\]
The following two lemma are useful in bounding the error terms in later sections.

\begin{lemma}\label{L:E_hard}
    With $E$, $A$, and $B$ as above and let $K > 0$ be any constant. Then
    \[
        \sum_{d \le x^{1/2}(\log x)^{-B}} \mu^2(d) K^{\go(d)} E(x,d) \ll \frac{x}{(\log x)^{(A-K^2)/2}}.
    \]
\end{lemma}

\begin{proof}
    The proof is an exercise in utilizing the Cauchy--Schwarz inequality. See the proof of~\cite[Lemma 8.1]{DHR}.
\end{proof}

The next lemma provides a trivial estimate on the error term without using the Bombieri--Vinogradov theorem.

\begin{lemma}\label{L:E_easy}
    Suppose that~\eqref{E:gO_condition} hold, and that
    \[
        \abs{R_d} \le \nu(d)
    \]
    whenever $d|P(z)$. Let $K > 0$, then
    \[
        \sum_{\substack{d < z^2\\ d|P(z)}} K^{\go(d)} |R_d| \le \frac{z^2}{V(z)^K},
    \]
    where
    \[
        V(z) = \prod_{p < z} \biggl( 1 - \frac{\nu(p)}{p} \biggr).
    \]
\end{lemma}

\begin{proof}
    The proof relies on the use of  Rankin's trick. See the proof of~\cite[Lemma 4.3]{DHR}.
\end{proof}

\section{Richert's weighted sieve}

Let $0 < \ga < \gb < \frac{1}{k}$ be parameters to be chosen later. We write $z = N^\ga$ and $y = N^\gb$, and let $r$ be a positive integer such that $r+1 > \frac{1}{\gb}$. Define
\[
    \eta = r + 1 - \frac{1}{\gb},
\]
then $\eta > 0$. Following Richert~\cite{Ric} and Irving~\cite{Irv}, we consider the weighted sum
\begin{equation}\label{E:W}
    W = W(\scrA,r,\ga,\gb) = \sum_{\substack{n \in \scrA\\ (n,P(z))=1}} \biggl( 1 - \frac{1}{\eta} \sum_{\substack{z \le p < y\\ p|n}} \biggl(1 - \frac{\log p}{\log y} \biggr) \biggr).
\end{equation}

We will show in Section 5 that
\begin{equation}\label{E:W_estimate}
    W \gg_{\ga,\gb,r} \frac{x}{(\log x)^2}.
\end{equation}
For now, we shall suppose that the above holds and prove that $\scrA$ contains $\gg \frac{x}{(\log x)^2}$ elements that are $P_r$-numbers.

First, we show that there is a suitable bound for the number of $n$ in $\scrA$ that are divisible by $p^2$ for $z \le p < y$. Observe that
\begin{align*}
    \sum_{z \le p < y} \sum_{\substack{x < q \le 2x\\ f(q) \equiv 0 \pmod{p^2}}} 1
        &\le \sum_{z \le p < y} \sum_{\substack{x < n \le 2x\\ f(n) \equiv 0 \pmod{p^2}}} 1\\
        &\ll_f \sum_{z \le p < y} \left( \frac{x}{p^2} + 1 \right) \le \frac{x}{z} + y = o\left( \frac{x}{(\log x)^2} \right).
\end{align*}
Thus the number of these elements are negligible given the estimate~\eqref{E:W_estimate} and can therefore be absorbed into the error term.

We now consider the elements $n$ in $\scrA$ that are free from the divisors $p^2$, where $z \le p < y$. If $(n,P(z)) = 1$, then $n$ makes a positive contribution to $W$ if and only if
\[
    1 - \frac{1}{\eta} \sum_{\substack{z \le p < y\\ p|n}} \left( 1 - \frac{\log p}{\log y} \right) > 0.
\]
This means that
\[
    \sum_{\substack{p < y\\ p|n}} \left( 1 - \frac{\log p}{\log y} \right) < r + 1 - \frac{1}{\gb}.
\]
On the other hand, it is still possible that $n$ has a repeated prime divisor $q$.
In this case, however, $q$ must be $\ge y$, and so we have
\[
    1 - \frac{\log q}{\log y} \le 0.
\]
Consequently, we have
\begin{align*}
    \sum_{\substack{p < y\\ p|n}} \left( 1 - \frac{\log p}{\log y} \right) + \sum_{\substack{q,a\\ q \ge y\\ q^a|n}} \left( 1 - \frac{\log q}{\log y} \right)< r+1-\frac{1}{\gb}.
\end{align*}
It then follows that
\begin{align*}
    \gO(n)
        &= \#\set{p|n : p < y} + \#\set{q^a|n : q \ge y}\\
        &< r + 1 - \frac{1}{\gb} + \frac{\log n}{\log y}\\
        &= r + 1 + \frac{1}{\gb} \biggl( \frac{\log n}{\log N} - 1 \biggr) \le r + 1.
\end{align*}
This shows that $\scrA$ contains $\gg \frac{x}{(\log x)^2}$ elements that are $P_r$-numbers, all of which have prime divisors $\ge z$.

\section{Solutions to Differential-Difference Equations}

We are interested in following continuous solutions $f_\gk$ and $F_\gk$ to certain differential-difference equations arising in sieve theory. These are the upper and lower sifting functions that occur in the DHR sieve method.

\begin{thm}\label{Th:DHRfF}~\cite[Theorem 6.1]{DHR}
    Let $\gk$ be a positive integer. Let $\s_k$ be the continuous solution to the system
    \begin{align}
        s^{-\gk} \s_{\gk}(s) &= \frac{(2e^\gr)^{-\gk}}{\gG(\gk+1)}, \quad 0 < s \le 2,\label{E:sigma_k}\\
        \frac{d}{ds}(s^{-\gk} \s_{\gk}(s)) &= -\gk s^{-\gk - 1} \s_{\gk}(s-2), \quad s > 2.\label{E:sigma_dde}
    \end{align}
    Then there exist $\ga_\gk, \gb_\gk \in \R^+$ such that the system
    \begin{align}
        F_\gk(s) &= \frac{1}{\s_\gk(s)}, \quad 0 < s \le \ga_\gk,\label{E:F_init}\\
        f_\gk(s) &= 0, \quad 0 < s \le \gb_\gk, \label{E:f_init}\\
        \frac{d}{ds}(s^{\gk} F_\gk(s)) &= \gk s^{\gk-1} f_\gk(s-1), \quad s > \ga_\gk,\label{E:ddF}\\
        \frac{d}{ds}(s^{\gk} f_\gk(s)) &= \gk s^{\gk-1} F_\gk(s-1), \quad s > \gb_\gk,\label{E:ddf}
    \end{align}
    with boundary conditions
    \[
        F_\gk(s) = 1 + O(e^{-s}) \quad \text{and} \quad f_\gk(s) = 1 + O(e^{-s})
    \]
    has continuous solutions $F = F_\gk$ and $f = f_\gk$ with the property that $F$ decreases monotonically and $f$ increases monotonically on $(0,\oo)$.
\end{thm}

When $\gk = 1$, it is known that $\ga_1 = \gb_1 = 2$~\cite[Table 17.1]{DHR}. It is also known that
\[
    f_1(s) =
    \begin{cases}
        0 &s \le 2,\\
        \dfrac{2e^\gr \log(s-1)}{s} &2 < s \le 4;
    \end{cases}
\]
while for $0 < s \le 3$,
\[
    F_1(s) = \dfrac{2e^\gr}{s}.
\]
Using~\eqref{E:ddF}, we see that for $3 < s \le 5$,
\begin{align*}
    F_1(s) &= \frac{1}{s} \biggl( 3F_1(3) + \int_3^s \frac{2e^\gr \log(t-2)}{t-1} \, dt \biggr)\\
           &= \frac{2e^\gr}{s} \biggl( 1 + \int_3^s \frac{\log(t-2)}{t-1}, dt \biggr).
\end{align*}
For $4 < s \le 6$, it follows at once from~\eqref{E:ddf} that
\begin{align*}
    f_1(s) &= \frac{1}{s} \biggl( 4f_1(4) + \int_4^s F_1(t-1) \, dt \biggr)\\
           &= \frac{2e^\gr}{s} \left\{ \log(s-1) + \int_4^s \int_3^t \frac{\log(u-2)}{u-1} \, du \, dt \right\}.
\end{align*}
We pause to remark that for $\gk = 1$, there is no essential difference between the beta sieve and the DHR sieve.

For $\gk = 2$, it is known that $\ga_2 = 5.3577...$ and $\gb_2 = 4.2664...$. Using~\eqref{E:sigma_k} and~\eqref{E:sigma_dde}, one could deduce similarily to $F_1$, that
\[
    F_2(s) =
    \begin{cases}
        \dfrac{8e^{2\gr}}{s^2}, &0 < s \le 2,\\
        \dfrac{4e^{2\gr}}{2(s-1)^2-s^2\log(\frac{s}{2})}, &2 < s \le 4,\\
        \displaystyle \left[ \frac{(9 - 8\log 2)s^2}{32 e^{2\gr}} - \frac{s^2}{2e^{2\gr}} \int_4^s \frac{2(t-3)^2 - (t-2)^2 \log(\frac{t-2}{2})}{t^3} \, dt \right]^{-1}, &4 < s \le \ga_2.
    \end{cases}
\]
We refrain from writing down $F_2(s)$ explicitly for $s > \ga_2$ for it involves a complicated expression. The evaluation of $F_2(s)$ for $s > \ga_2$ will require numerical techniques.

\section{Sieve Estimates}

Recall that we seek to estimate the sum $W$. To ease notation, we follow Irving and write
\[
    w_p = 1 - \frac{\log p}{\log y}.
\]
$W$ can then be written as
\[
    W = S(\scrA,z) - \frac{1}{\eta} \sum_{z \le p < y} w_p S(\scrA_p,z).
\]
Let $\gd \in (\ga, \gb)$ and denote $u = N^\gd$, where $\gd < \frac{1}{2k}$. Then $W$ may be written as
\[
    W = S(\scrA,z) - \frac{1}{\eta} (S_1^\ast + S_2^\ast),
\]
where
\[
    S_1^\ast = \sum_{z \le p < u} w_p S(\scrA_p,z) \quad \text{and} \quad S_2^\ast = \sum_{u \le p < y} w_p S(\scrA_p,z).
\]

We consider the linearized problem for $S(\scrA,z)$ and $S_1^\ast$ in order to apply the
Bombieri--Vinogradov theorem, and the non-linearized problem for $S_2^\ast$ for the elements outside
of the level of distribution applicable to the Bombieri--Vinogradov theorem.

Since $N = \max \scrA$, we have $N \ll x^k$ and so $N^{1/k} \ll x$. This suggests that the
appropriate level of distribution for $S(\scrA,z)$ is $\gth_1 < \frac{1}{2k}$. For $S_1^\ast$, our sequence
is $\scrA_p$, thus the level of distribution is $N^{\gth_1}/p$.

For $S_2^\ast$, the sifting sequence is $\scrA_p'$, where
\[
	\scrA'= \set{ nf(n) : n \in (x,2x] };
\]
we apply a two-dimensional sieve to $\scrA_p'$ with $N^{\gth_2}/p$, where $\gth_2 < \frac{1}{k}$. Then the level of distribution is $N^{\gth_2}/p$.

First, we apply the one-dimensional DHR sieve to $S(\scr{A},z)$. By~\cite[Theorem 7.1]{DHR}, we have
\begin{equation}\label{E:S(A,z)}
	S(\scr{A},z)
        \ge XV_1(z) \biggl\{ f_1\biggl(\frac{\gth_1}{\ga}\biggr) - O \biggl( \frac{(\log \log N^{\gth_1})^{3/4}}{(\log N^{\gth_1})^{1/4}} \biggr) \biggr\} - \sum_{\substack{d < N^{\gth_1}\\ d|P(z)}} 4^{\go(d)} |R_d|,
\end{equation}
where
\[
	V_1(z) = \prod_{p < z} \biggl( 1 - \frac{\nu_1(p)}{\gf(p)} \biggr).
\]
Observe that
\[
    \frac{(\log \log N^{\gth_1})^{3/4}}{(\log N^{\gth_1})^{1/4}} \ll \frac{1}{(\gth_1 \log N)^{1/8}} \ll_{\gth_1, f} \frac{1}{(\log x)^{1/8}} = o(1).
\]
The sum on the right-hand side of~\eqref{E:S(A,z)} is bounded by
\begin{equation}\label{E:E1_bound}
    \sum_{\substack{d < N^{\gth_1}\\ d|P(z)}} 4^{\go(d)} \abs{R_d} \le \sum_{\substack{d < N^{\gth_1}\\ d|P(z)}} (4k)^{\go(d)} E(x,d) + \sum_{\substack{d < N^{\gth_1}\\ d|P(z)}} (4k)^{\go(d)}.
\end{equation}
The first sum on the right-hand of~\eqref{E:E1_bound} can be estimated using Lemma~\ref{L:E_hard}; we have
\[
    \sum_{\substack{d < N^{\gth_1}\\ d|P(z)}} (4k)^{\go(d)} E(x,d) \ll_{\gth_1} \frac{x}{(\log x)^{A}}.
\]
Next, we use Lemma~\ref{L:E_easy} to bound the second sum on the right-hand of~\eqref{E:E1_bound}
\[
    \sum_{\substack{d < N^{\gth_1}\\ d|P(z)}} (4k)^{\go(d)} \le N^{\gth_1} (\log z)^{4k}.
\]
Observe that the first estimate dominates the second one. Moreover,
\[
    V_1(z) \gg \frac{1}{\log z}
\]
and so the error term is
\[
    \sum_{\substack{d < N^{\gth_1}\\ d|P(z)}} 4^{\go(d)} |R_d| = o(XV_1(z))
\]
Equation \eqref{E:S(A,z)} may therefore be written as
\begin{equation}\label{E:S(A,z)2}
    S(\scrA, z) \ge XV_1(z) \biggl\{ f_1\biggl(\frac{\gth_1}{\ga}\biggr) - o(1) \biggr\}.
\end{equation}

Applying the one-dimensional upper bound sieve, we arrive at the following lemma.

\begin{lemma}
    If $0 < \ga < \gd < \gth_1$, then
    \[
        S_1^\ast \le XV_1(z) \biggl\{ \int_\ga^\gd \biggl( \frac{1}{s} - \frac{1}{\gb} \biggr) F_1\biggl(\frac{\gth_1-s}{\ga}\biggr) \, ds + o(1) \biggr\}.
    \]
\end{lemma}

\begin{proof}
    We apply the one-dimensional upper bound DHR sieve with level of distribution $N^{\gth_1}/p$ to each $p$ to obtain
    \begin{align*}
      S(\scrA_p,z) \le \frac{\nu_1(p)}{\gf(p)} XV_1(z) \biggl\{ F\biggl( \frac{\log(N^{\gth_1}/p)}{\log z} \biggr) + O\biggl( &\frac{(\log\log(N^{\gth_1}/p))^{3/4}}{(\log(N^{\gth_1}/p))^{1/4}} \biggr) \biggr\}\\
        &\phantom{blahblahblah}+ \sum_{\substack{d < N^{\gth_1}/p\\ d|P(z)}} 4^{\go(d)} |R_{pd}| .
    \end{align*}
    Since $p < N^\gd$ and $\gd < \gth_1$, it follows that
    \begin{align*}
        \frac{(\log \log(N^{\gth_1}/p))^{3/4}}{\log(N^{\gth_1}/p)^{1/4}}
            &\ll \frac{1}{\log(N^{\gth_1}/p)^{1/8}}\\
            &\le \frac{1}{((\gth_1-\gd) \log N)^{1/8}} \ll_{\gth_1,\gd,f} \frac{1}{(\log x)^{1/8}}.
    \end{align*}

    For the error term, we appeal to Lemma~\ref{L:E_hard} to get
    \begin{align*}
        \sum_{z \le p < u} w_p \sum_{\substack{m < N^{\gth_1}/p\\ m|P(z)}} 4^{\go(m)} |R_{pm}|
            &\le \sum_{d \le N^{\gth_1}} \mu^2(d) 4^{\go(d)} |R_d|\\
            &\ll_{\gth_1} \frac{x}{(\log x)^4} = o(XV_1(z)).
    \end{align*}

    Denote
    \[
        s_p = \frac{\log(N^{\gth_1}/p)}{\log z}.
    \]
    Then
    \begin{equation}
        \sum_{z \le p < u} w_p S(\scrA_p,z) \le XV_1(z) \sum_{z \le p < u} \frac{\nu_1(p)}{\gf(p)} w_p \biggl( F_1(s_p) + O\biggl( \frac{1}{(\log x)^{\frac{1}{8}}} \biggr) \biggr) + o(XV_1(z)).
    \end{equation}
    Note that since $0 < \ga < \gd$, it follows that
    \[
        1 - \frac{\log p}{\log u} \le 1 - \frac{\log z}{\log u} = 1 - \frac{\ga}{\gb} < 1,
    \]
    and so
    \begin{align*}
        \sum_{z \le p < u} \frac{\nu_1(p)}{\gf(p)} w_p (\log x)^{-\frac{1}{8}}
            &< (\log x)^{-\frac{1}{8}} \sum_{z \le p < u} \frac{\nu_1(p)}{p-1}\\
            &\ll (\log x)^{-\frac{1}{8}} \sum_{z \le p < u} \frac{1}{p} \ll \frac{\log \log u}{(\log x)^{\frac{1}{8}}} = o(1).
    \end{align*}
	For the main term in (17), define
	\[
		g(t) := \biggl( \frac{1}{\log t} - \frac{1}{\log y} \biggr) F_1\biggl(
        \frac{\log(N^{\gth_1}/t)}{\log z} \biggr) \quad \text{and} \quad C(z,t) := \sum_{z \le p < t} \frac{\nu_1(p) \log p}{p-1}.
	\]
	Partial summation gives us
	\begin{align*}
		\sum_{z \le p < u} w_p \frac{\nu_1(p)}{\gf(p)} F_1(s_p)
			&= \sum_{z \le p < u} \frac{\nu_1(p) \log p}{p-1} g(p)\\
			&= \int_z^u g(t) \, dC(z,t)\\
            &= g(u) C(z,u) - \int_z^u C(z,t) g'(t) \, dt.
	\end{align*}
    We use~\eqref{E:nu_1_sum} to obtain
    \[
        g(u) C(z,u) = g(u)(\log u - \log z + O(1)).
    \]
    Note that
    \[
        g(u) = g(N^\gd) = \biggl( \frac{1}{\gd \log N} - \frac{1}{\gb \log N} \biggr) F_1\biggl( \frac{\gth_1-\gd}{\ga} \biggr) \ll \frac{1}{\log N} = o(1).
    \]

	Next, let $s = \frac{\log t}{\log N}$, then
    \[
      	\int_z^u |g'(t)| \, dt = \int_\ga^\gd \biggl| \frac{dg(N^s)}{N^s \log N \, ds} \biggr| N^s
      \log N \, ds = \int_\ga^\gd  \biggl| \frac{d g(N^s)}{ds} \biggr| \, ds.
    \]
	However,
	\[
		g(N^s) = \Bigl( \frac{1}{s \log N} - \frac{1}{\gb \log N} \Bigr) F_1\Bigl( \frac{\gth_1 -
          s}{\ga} \Bigr)
	\]
    and that
    \begin{align*}
        \frac{dg(N^s)}{ds}
            &= \frac{-1}{s^2 \log N} F_1\Bigl(\frac{\gth_1-s}{\ga}\Bigr) - \frac{1}{\ga s \log N} F_1'\Bigl( \frac{\gth_1 - s}{\ga} \Bigr) + \frac{1}{\ga \gb \log N} F_1' \Bigl( \frac{\gth_1 - s}{\ga} \Bigr)\\
            &= \frac{-\ga \gb F_1( \frac{\gth_1-s}{\ga} ) + s(s-\gb) F_1'( \frac{\gth_1-s}{\ga} ) }{\ga \gb s^2 \log N} \ll \frac{1}{\log N} = o(1).
    \end{align*}
    Therefore
    \[
        \int_z^u C(z,t) g'(t) \, dt = \int_z^u [ \, \log u - \log z + O(1)] \, g'(t) \, dt,
    \]
    and so
    \begin{align*}
        \sum_{z \le p < u} w_p \frac{\nu_1(p)}{\gf(p)} F_1(s_p)
            &= g(u) C(z,u) - \int_z^u C(z,t) g'(t) \, dt\\
            &= (\log u - \log z + O(1))g(u) - \int_z^u (\log u - \log z)g'(t) \, dt - \int_z^u O(g'(t)) \, dt\\
            &= (\log u - \log z)g(u) - \int_z^u (\log u - \log z) g'(t) \, dt + o(1)\\
            &= \int_z^u \frac{g(t)}{t} \, dt + o(1)\\
            &= \int_z^u \frac{1}{t} \Bigl( \frac{1}{\log t} - \frac{1}{\log y} \Bigr) F_1\Bigl( \frac{\log(N^{\gth_1}/t)}{\log z} \Bigr) \, dt + o(1)\\
            &= \int_\ga^\gd \Bigl( \frac{1}{s} - \frac{1}{\gb} \Bigr) F_1\Bigl( \frac{\gth_1 - s}{\ga} \Bigr) \, ds + o(1).
    \end{align*}
    This concludes the proof of the lemma.
\end{proof}

The estimate for $S_2^\ast$ is more involved. We first provide an estimate for $S(\scrA_p,z)$ using a two-dimensional sieve with level of distribution $N^{\gth_2}/p$, where $p \in [z,y)$ and $\gth_2 < \frac{1}{k}$.

\begin{lemma}
    Suppose $\ga < \frac{1}{k}$ and $\gb < \gth_2 < \frac{1}{k}$. Then for any $p$ such that $p \in [z,y)$, we have
    \[
        S(\scrA_p,z) \le \frac{x\nu_1(p)}{p} V_2(z)\biggl( F_2(s'_p) + O\biggl(\frac{1}{(\log x)^{1/8}}\biggr) \biggr),
    \]
    where
    \[
        V_2(z) = \prod_{p < z} \biggl( 1 - \frac{\nu_2(p)}{p} \biggr) \quad \text{and} \quad s_p' = \frac{\log(N^{\gth_2}/p)}{\log z}.
    \]
\end{lemma}

\begin{proof}
    Suppose that $p \ge z$, then
    \begin{align}\label{E:2d_S(Ap,z)}
        S(\scrA_p,z) 
                     &= \# \{ q \in (x,2x] : p|f(q), (qf(q), P(z)) = 1 \} \notag\\
                     &\le \# \{ n \in (x,2x] : p|f(n), (nf(n),P(z)) = 1 \}.
    \end{align}
    If we let
    \[
        \scrA' = \set{ nf(n) : n \in (x,2x] }.
    \]
    Then the inequality~\eqref{E:2d_S(Ap,z)} tells us that we can apply an upper bound sieve to the sequence $\scrA'_p$. Let $d|P(z)$ and $p$ sufficiently large so that $p \nmid f(0)$. Then $(d,p) = 1$ and by the Chinese Remainder Theorem, we have
    \[
        \# \{ a \pmod{dp} : f(a) \equiv 0 \pmod{p}, af(a) \equiv \pmod{d} \} = \nu_1(p) \nu_2(d).
    \]
    Also, for any $\ve > 0$,
    \[
        \nu_1(p) \ll 1 \quad \text{and} \quad \nu_2(d) \ll_\ve d^\ve.
    \]
    Consequently, we have
    \begin{align*}
        \# \{ n \in (x,2x] : p|f(n), d|nf(n) \}
            = \sum_{\substack{a \pmod{p}\\ f(a) \equiv 0 \pmod{p}\\ af(a) \equiv 0 \pmod{d}}} \biggl( \frac{x}{pd} + O(1) \biggr)
            = \frac{\nu_1(p)}{p} \frac{\nu_2(d)}{d} x + O(d^\ve).
    \end{align*}
    We apply the two-dimensional DHR sieve~\cite[Lemma 9.3]{DHR} with level of distribution $N^{\gth_2}/p$ to get
    \begin{align*}
        S(\scrA_p,z) \le \frac{x \nu_1(p)}{p} V_2(z) \biggl\{ F_2\biggl( \frac{\log (N^{\gth_2}/p)}{\log z} \biggr) &+ O\biggl( \frac{(\log \log N^{\gth_2})^2}{(\log (N^{\gth_2}/p))^{1/6}} \biggr) \biggr\}\\
        &\phantom{blahblah} + 2 \sum_{\substack{d < N^{\gth_2}/p\\ d|P(z)}} 4^{\go(d)} |R_{pd}|.
    \end{align*}
    First note that since $p < N^\gb$ and $\gb < \gth_2$, it follows that
    \[
        \frac{(\log \log N^{\gth_2})^2}{(\log (N^{\gth_2}/p))^{1/6}} \ll \frac{1}{(\log (N^{\gth_2}/p))^{1/8}} \ll_{\gth_2,f} \frac{1}{(\log x)^{1/8}}.
    \]
    The error term can be estimated as following using Lemma~\ref{L:E_easy}.
    \[
        \sum_{\substack{d < N^{\gth_2}/p\\ d|P(z)}} 4^{\go(d)} \abs{R_{pd}}
            \le \frac{N^{\gth_2}}{p} V_2(z)^{-4}
            \ll_\ve \frac{x^{1-\ve}}{p} \frac{V_2(z)}{V_2(z)^5} \ll \frac{x}{p} \frac{V_2(z)}{(\log x)^{1/8}}. \qedhere
    \]
\end{proof}

\begin{lemma}
    Let $\ga$, $\gb$, and $\gth_2$ satisfy the hypotheses of the previous lemma. If $\ga < \gd < \gb$, then
    \[
        S_2^\ast \le XV_1(z) \biggl\{ \frac{e^{-\gr}}{k\ga} \int_\ga^\gb \biggl( \frac{1}{s} - \frac{1}{\gb} \biggr) F_2\biggl( \frac{\gth_2-s}{\ga} \biggr)  ds + o(1) \biggr\}
    \]
\end{lemma}

\begin{proof}
    By Lemma 5.2, we immediately have
    \[
        S_2^\ast \le \sum_{u \le p < y} w_p \frac{x\nu_1(p)}{p} V_2(z) \biggl( F_2(s_p') + O \biggl( \frac{1}{(\log x)^{1/8}}\biggr) \biggr).
    \]
    The error term can be bounded by noting that
    \[
        \sum_{u \le p < y} w_p \frac{\nu_1(p)}{p} \frac{1}{(\log x)^{1/8}} \ll \frac{1}{(\log x)^{1/8}} \sum_{u \le p < y} \frac{1}{p} = o(1).
    \]
    For the main term, let
    \[
        h(t) = \biggl( \frac{1}{\log t} - \frac{1}{\log y} \biggr) F_2\biggl( \frac{\log(N^{\gth_2}/p)}{\log z} \biggr)
    \]
    and
    \[
        D(u,t) := \sum_{u \le p < t} \frac{\nu_1(p) \log p}{p}.
    \]
    From the proof of Lemma 5.1, we know that $h(y)$ and $\int_u^y O(h'(t)) \, dt$ are both $o(1)$. We may thus use partial summation and~\eqref{E:nu_1_sum} to see that
    \begin{align}\label{E:S_2*}
        \sum_{u \le p < y} w_p \frac{\nu_1(p)}{p} F_2(s_p')
            &= \sum_{u \le p < y} \frac{\nu_1(p) \log p}{p} h(p)\notag\\
            &= \int_u^y h(t) \, dD(u,t)\notag\\
            &= (\log y - \log u)h(y) - \int_u^y (\log y - \log u) h'(t) \, dt + o(1)\notag\\
            &= \int_u^y \frac{h(t)}{t} \, dt + o(1)\notag\\
            &= \int_u^y \frac{1}{t} \biggl( \frac{1}{\log t} - \frac{1}{\log y} \biggr) F_2\biggl( \frac{\log(N^{\gth_2}/t)}{\log z} \biggr) \, dt + o(1)\notag\\
            &= \int_{\gd}^{\gb} \biggl( \frac{1}{s} - \frac{1}{\gb} \biggr) F_2\biggl( \frac{\gth_2 - s}{\ga} \biggr) \, ds + o(1).
    \end{align}
    Next, we observe that by~\eqref{E:nu_2_prod},
    \begin{equation}\label{E:V2}
        V_2(z) = \frac{V_1(z)}{\log z}(e^{-\gr} + o(1));
    \end{equation}
    and by the Prime Number Theorem,
    \begin{equation}\label{E:pnt}
        x \sim X \log x \sim \frac{1}{k\ga} X \log z.
    \end{equation}
    The lemma follows by combining~\eqref{E:S_2*} with~\eqref{E:V2} and~\eqref{E:pnt}.
\end{proof}

Using equation \eqref{E:S(A,z)2} and Lemmas 5.1 and 5.3, we have a lower estimate for $W$; namely
\begin{align*}
    W &= S(\scrA,z) - \frac{1}{\eta} (S_1^\ast + S_2^\ast)\\
      &\ge XV_1(z) \left\{ f_1\biggl(\frac{\gth_1}{\ga}\biggr) - \frac{1}{\eta} \biggl[ \int_\ga^\gd \biggl( \frac{1}{s} - \frac{1}{\gb} \biggr) F_1\biggl( \frac{\gth_1-s}{\ga} \biggr) \, ds \right.\\
      &\phantom{blahblahblahblahbla} + \left. \frac{e^{-\gr}}{k\ga} \int_\gd^\gb \left( \frac{1}{s} - \frac{1}{\gb} \right) F_2 \left( \frac{\gth_2 - s}{\ga} \right) \, ds + o(1) \biggr] \right\},
\end{align*}
provided that $0 < \ga < \gd < \gb < \gth_2 < \frac{1}{k}$ and $\gd < \gth_1 < \frac{1}{2k}$. By continuity of the sifting functions, for any $\ve > 0$, we may take $\gth_1$ and $\gth_2$ sufficiently close to $\frac{1}{2k}$ and $\frac{1}{k}$, respectively. Furthermore, the Prime Number Theorem gives us
\[
    XV_1(z) \gg \frac{x}{(\log x)^2}.
\]
We thus have the following lower bound for $W$.
\begin{align}\label{E:asymp_W}
    W &\gg \frac{x}{(\log x)^2} \left\{ f_1\biggl(\frac{1}{2k\ga}\biggr) - \frac{1}{r+1-\frac{1}{\gb}} \biggl[ \, \int_\ga^\gd \biggl( \frac{1}{s} - \frac{1}{\gb} \biggr) F_1\biggl( \frac{1-2ks}{2k\ga} \biggr) \, ds \right.\\
    &\phantom{blahblahblahblahbla} + \left. \frac{e^{-\gr}}{k\ga} \int_\gd^\gb \left( \frac{1}{s} - \frac{1}{\gb} \right) F_2 \left( \frac{1-ks}{k\ga} \right) \, ds - \ve + o(1) \biggr] \right\},\notag
\end{align}
in which we replaced $\eta$ with its definition $\eta = r + 1 - \frac{1}{\gb}$. It is now clear that one can deduce the results of Theorem~\ref{T:1} provided that one can find the suitable $\ga$, $\gd$, and $\gb$ satisfying
\[
    0 < \ga < \gd < \gb < \frac{1}{k}, \quad \gd < \frac{1}{2k}, \quad \text{and} \quad \gb > \frac{1}{r+1}
\]
such that
\begin{align*}
    f_1\biggl(\frac{1}{2k\ga}\biggr)
        &- \frac{1}{r+1-\frac{1}{\gb}} \biggl[ \, \int_\ga^\gd \biggl( \frac{1}{s} - \frac{1}{\gb} \biggr) F_1\biggl( \frac{1-2ks}{2k\ga} \biggr) \, ds\\
        &\phantom{blahblahblahblah}+ \frac{e^{-\gr}}{k\ga} \int_\gd^\gb \left( \frac{1}{s} - \frac{1}{\gb} \right) F_2 \left( \frac{1-ks}{k\ga} \right) \, ds \biggr] > 0.
\end{align*}
To ease notation, we let
\[
    \ga_0 = k\ga, \quad \gd_0 = k\gd, \quad \text{and} \quad \gb_0 = k\gb,
\]
and make the observation that the above is the equivalent of
\begin{align}\label{E:r_ineq}
	r &> \frac{k}{\gb_0} - 1 + \frac{1}{f_1(1/2\ga_0)} \biggl[ \, \int_{\ga_0}^{\gd_0} \biggl(
      \frac{1}{s} - \frac{1}{\gb_0} \biggr) F_1\biggl( \frac{1-2s}{2\ga_0} \biggr) \, ds\\
      &\phantom{blahblahblahblahblahblah} + \frac{e^{-\gr}}{\ga_0} \int_{\gd_0}^{\gb_0} \biggl( \frac{1}{s} - \frac{1}{\gb_0} \biggr) F_2\biggl( \frac{1-s}{\ga_0} \biggr) \, ds \biggr].\notag
\end{align}
In the next section, we complete the proof of Theorem~\ref{T:1} by finding the optimal values of $\ga_0$, $\gd_0$, and $\gb_0$ which will minimize $r$ with respect to a given $k$, where they also satisfy the conditions
\[
    0 < \ga_0 < \gd_0 < \gb_0 < 1, \quad \gd_0 < \frac{1}{2}, \quad \text{and} \quad \frac{\gb_0}{k} > \frac{1}{r+1}.
\]

\section{Proof of Theorem~\ref{T:1}}

In order to improve upon Irving's results, one must extend the ranges of the sifting functions $f_1$, $F_1$, and $F_2$ which appeared on the right-hand side of~\eqref{E:r_ineq}. In particular, we will consider $f_1(s)$ in the range of $2 \le s \le 6$, $F_1(s)$ in the range of $0 < s \le 5$, and $F_2(s)$ in the range of $0 < s < 7$. We remark all numerical computations of $f_1$, $F_1$, and $F_2$ are completed in {\it Mathematica} using the packaged developed by W.\ Galway~\cite{Gal}.

To this end, we take $\ga_0 = \frac{1}{12}$, so that $f_1(\frac{1}{2\ga_0}) = f_1(6)$. This immediately gives us
\[
    \frac{1-2s}{2\ga_0} = 6 - 12s \quad \text{and} \quad \frac{1-s}{\ga_0} = 12 - 12s.
\]
Next, observe that
\[
    0 < 6 - 12\gd_0 \le 3 \quad \text{if and only if} \quad \frac{1}{4} \le \gd_0 < \frac{1}{2};
\]
and
\[
	3 < 6 - 12s \le 5 \quad \text{if and only if} \quad \frac{1}{12} \le s < \frac{1}{4}.
\]
Therefore $F_1(6-12s)$ is non-elementary for $\ga_0 = \frac{1}{12} \le s < \frac{1}{4}$ and is elementary for $\frac{1}{4} \le s < \gd_0$. Similarly,
\[
	12 - 12s \le 2 \quad \text{if and only if} \quad s \ge \frac{5}{6},
\]
and it follows that $F_2(12-12s)$ is non-elementary for $\gd_0 \le s < \frac{5}{6}$ and is elementary for $s \ge \frac{5}{6}$. 
We seek the optimal choice of $\gd_0$ that minimizes the right-hand side of~\eqref{E:r_ineq}. This can be accomplished by finding the $\gd_0$ that minimizes the quantity inside the brackets of the inequality~\eqref{E:r_ineq}. In other words, we solve the equation
\[
    \frac{\del}{\del \gd_0} \biggl[ \, \int_{\frac{1}{12}}^{\gd_0} \biggl( \frac{1}{s} - \frac{1}{\gb_0} \biggr) F_1(6-12s) \, ds + \frac{12}{e^{\gr}} \int_{\gd_0}^{\gb_0} \biggl( \frac{1}{s} - \frac{1}{\gb_0} \biggr) F_2(12-12s) \, ds \biggr] = 0
\]
for $\gd_0$. Applying the Fundamental Theorem of Calculus, one sees that solving the above equation is equivalent to solving
\[
	\frac{e^\gr}{3} \frac{1}{1-2\gd_0} = \frac{12}{e^\gr} F_2(12 - 12\gd_0).
\]
Numerical computation yields $\gd_0 = 0.45804...$.

Observe that~\eqref{E:r_ineq} implies that
\[
    r > \frac{k}{\gb_0} - 1.
\]
In other words,
\[
    \frac{\gb_0}{k} > \frac{1}{r+1}.
\]
We now proceed to find the admissible $\gb_0$ that will minimize the right-hand side of~\eqref{E:r_ineq}. To do this, we first set
\begin{align}\label{E:rkgb_0}
    r(k,\gb_0) = \frac{k}{\gb_0} - 1 + \frac{1}{f_1(6)} \biggl[ \, \int_{\frac{1}{12}}^{\gd_0} &\biggl( \frac{1}{s} - \frac{1}{\gb_0} \biggr) F_1(6-12s) \, ds \biggr.\\
    &\biggl. + \frac{12}{e^\gr} \int_{\gd_0}^{\gb_0} \biggl( \frac{1}{s} - \frac{1}{\gb_0} \biggr) F_2(12 - 12s) \, ds \biggr].\notag
\end{align}
Taking the partial derivative of $r(k,\gb_0)$ with respect to $\gb_0$, we get
\begin{equation}\label{E:del_r}
    \frac{\del}{\del \gb_0} r(k,\gb_0) = -\frac{k}{\gb_0^2} + \frac{1}{f_1(6)} \biggl[ \, \frac{1}{\gb_0^2} \int_{\frac{1}{12}}^{\gd_0} F_1(6-12s) \, ds + \frac{12}{e^\gr} \frac{1}{\gb_0^2} \int_{\gd_0}^{\gb_0} F_2(12 - 12s) \, ds \biggr].
\end{equation}
For $k = 2, \ldots, 6$, we set~\eqref{E:del_r} to zero and simplify to get
\[
    \int_{\gd_0}^{\gb_0} F_2(12 - 12s) \, ds = \frac{e^\gr}{12} \biggl[ f_1(6)k - \int_{\frac{1}{12}}^{\gd_0} F_1(6 - 12s) \, ds \biggr].
\]
Numerical computations suggest that the optimal $\gb_0$'s are less than $\frac{5}{6}$. Therefore $F_2(12-12s)$ is non-elementary in the interval $[\gd_0, \gb_0)$. We approxmiate the optimal values of $\beta_0$ numerically and obtain the desired values for $r(k,\gb_0)$.

\begin{center}
\begin{table}[h]
\renewcommand{\arraystretch}{1.2}
\begin{tabular}{@{}clllll@{}}\toprule
	         $k$ &      2 &      3 &      4 &      5 &      6 \\ \hline
	     $\gb_0$ & 0.6131 & 0.6968 & 0.7552 & 0.7969 & 0.8265 \\ \hline
	$r(k,\gb_0)$ & 3.9667 & 5.4803 & 6.8645 & 8.1510 & 9.3819 \\ \bottomrule
\end{tabular}
\label{Table:2}
\vspace{12pt}
\caption{$r(k,\gb_0)$ for $k = 2, \ldots, 6$}
\end{table}
\end{center}

For $k \ge 7$, $\gb_0$ will be greater than $\frac{5}{6}$. In this case, we may  write~\eqref{E:rkgb_0} as
\begin{align*}
	r(k,\gb_0) = \frac{k}{\gb_0} - 1 &+ \frac{1}{f_1(6)} \biggl[ \, \int_{\frac{1}{12}}^{\frac{1}{4}} \biggl( \frac{1}{s} - \frac{1}{\gb_0} \biggr) F_1(6 - 12s) \, ds + \frac{e^\gr}{3} \int_{\frac{1}{4}}^{\gd_0} \biggl( \frac{1}{s} - \frac{1}{\gb_0} \biggr) \frac{ds}{1-2s} \biggr.\\
    & \biggl. + \frac{12}{e^\gr} \int_{\gd_0}^{\frac{5}{6}} \biggl( \frac{1}{s} - \frac{1}{\gb_0} \biggr) F_2(12-12s) \, ds + \frac{2e^{\gr}}{3} \int_{\frac{5}{6}}^{\gb_0} \biggl( \frac{1}{s} - \frac{1}{\gb_0} \biggr) \frac{ds}{(1-s)^2} \biggr].\notag
\end{align*}
Moreover, equation~\eqref{E:del_r} becomes
\begin{align}
	\frac{\del}{\del \gb_0} r(k,\gb_0) = -\frac{k}{\gb_0^2} + \frac{1}{f_1(6)} \biggl[
      	\frac{1}{\gb_0^2} \int_{\frac{1}{12}}^{\frac{1}{4}} &F_1(6-12s) \, ds - \frac{e^\gr}{6\gb_0^2} \log(2-4\gd_0)\\
		& + \frac{1}{\gb_0^2} \frac{12}{e^\gr} \int_{\gd_0}^{\frac{5}{6}} F_2(12-12s) \, ds + \frac{2e^\gr}{3} \frac{6\gb_0-5}{(1-\gb_0)\gb_0^2} \biggr].\notag
\end{align}
We note that
\[
    f_1(6) = 0.99989...,
\]
and write
\begin{align*}
    M_1 &= \frac{3}{2e^\gr} = 0.84218...,
	&&M_2 = \int_{\frac{1}{12}}^{\frac{1}{4}} F_1(6-12s) , ds = 0.17383...,\\
	M_3 &= -\frac{e^\gr}{6} \log(2-4\gd_0) = 0.52979...,
	&&M_4 = \frac{12}{e^\gr}\int_{\gd_0}^{\frac{5}{6}} F_2(12-12s) \, ds = 5.57453....
\end{align*}
Then we may set~\eqref{E:del_r} to zero and simplify to obtain
\begin{equation}\label{E:delr0}
	k = \frac{1}{f_1(6)} \left[ M_2 + M_3 + M_4 + \frac{1}{M_1} \frac{6\gb_0-5}{1-\gb_0} \right].
\end{equation}
Continuing to solve for $\gb_0$ yields
\[
	\gb_0 = 1 - \frac{1}{c_1 k + c_2},
\]
where,
\[
	c_1 = M_1 f_1(6) = 0.842101... \quad \text{and} \quad
    c_2 = -M_1(M_2 + M_3 + M_4) + 6 = 0.712608....
\]
We summarize the values of $r(k,\gb_0)$ for small $k$ in Table~\ref{Table:3} below.

\begin{center}
\begin{table}[h]
\renewcommand{\arraystretch}{1.2}
\begin{tabular}{@{}clllllllll@{}}\toprule
	         $k$ &    2 &    3 &    4 &    5 &    6 &     7 &     8 &     9 &    10\\ \hline
	     $\gb_0$ & 0.61 & 0.69 & 0.75 & 0.79 & 0.82 &  0.85 &  0.87 &  0.88 &  0.89\\ \hline
	$r(k,\gb_0)$ & 3.96 & 5.48 & 6.86 & 8.15 & 9.38 & 10.58 & 11.74 & 12.89 & 14.02\\ \hline
             $r$ &    4 &    6 &    7 &    9 &   10 &    11 &    12 &    13 &    15\\ \bottomrule
\end{tabular}
\vspace{6pt}
\caption{Results for small $k$}
\label{Table:3}
\end{table}
\end{center}

Lastly, for large $k$, we will take
\[
	\gb_0 = 1 - \frac{1}{k}.
\]
This immediately gives
\[
	\frac{k}{\gb_0} = \frac{k^2}{k-1} = k + O(1),
\]
while it is clear that
\[
	\int_{\frac{1}{12}}^{\frac{1}{4}} \biggl( \frac{1}{s} - \frac{1}{\gb_0} \biggr) F_1(6-12s) \,ds,
    \quad \int_{\frac{1}{4}}^{\gd_0} \biggl( \frac{1}{s} - \frac{1}{\gb_0} \biggr) \frac{ds}{1-2s},
    \quad \text{and} \quad \int_{\gd_0}^{\frac{5}{6}} \biggl( \frac{1}{s} - \frac{1}{\gb_0} \biggr)
    F_2(12-12s) \, ds
\]
are all $\ll 1$. It remains to show that
\begin{align*}
  	\int_{\frac{5}{6}}^{\gb_0} \biggl( \frac{1}{s} - \frac{1}{\gb_0} \biggr) \frac{ds}{(1-s)^2}
		&= \int_{\frac{5}{6}}^{\gb_0} \left[ \frac{1}{s} + \frac{1}{1-s} + \left( 1 -
          \frac{1}{\gb_0} \right) \frac{1}{(1-s)^2} \right] \, ds\\
		&= \left[ -\log(1-s) + \left( 1 - \frac{1}{\gb_0} \right) \frac{1}{1-s}
          \right]_{\frac{5}{6}}^{\gb_0} + O(1)\\
		&= \log k + O(1).
\end{align*}
We therefore have
\[
	r(k,\gb_0) = k + c \log k + O(1),
\]
where
\[
	c = \frac{2e^{\gr}}{3f_1(6)} = 1.18751....
\]

\section{Acknowledgements}

I would like to thank my advisor, Sid Graham, for his careful reading of this paper and considerable assistance on the content and the style of this paper. In addition, I also would like to thank Craig Franze for many helpful conversations and suggestions.

\bibliographystyle{amsplain}

\end{document}